\documentclass[10pt,draft]{article}

\usepackage{amsmath,amssymb,amsthm}
\usepackage{color}
\usepackage{mathrsfs} 
\usepackage{fancyhdr} 
\usepackage{dsfont} 
\usepackage{enumerate} 
\usepackage[normalem]{ulem}
\usepackage{a4wide}
\newcommand{\D}{\mathbb{D}}

\newcommand{\N}{\mathbb{N}}

\newcommand{\R}{\mathbb{R}}

\newcommand{\GG}{\mathscr{G}}
\newcommand{\HH}{\mathscr{H}}
\newcommand{\LL}{\mathscr{L}}

\newcommand{\PP}{\mathscr{P}}

\newcommand{\vv}{{\mbox{\boldmath$v$}}}

\newcommand{\mmu}{{\mbox{\boldmath$\mu$}}}

\newcommand{\rrho}{{\mbox{\boldmath$\rho$}}}

\newcommand{\sfs}{{\sf s}}
\newcommand{\sft}{{\sf t}}

\newcommand{\sfA}{{\sf A}}

\newcommand{\sfX}{{\sf X}}

\newcommand{\rmb}{{\mathrm b}}
\newcommand{\rmc}{{\mathrm c}}

\newcommand{\rme}{{\mathrm e}}

\newcommand{\Kliminf}{K\kern-3pt-\kern-2pt\mathop{\rm lim\,inf}\limits}  
\newcommand{\supp}{\mathop{\rm supp}\nolimits}   
\renewcommand{\d}{{\mathrm d}}

\newcommand{\restr}[1]{\lower3pt\hbox{$|_{#1}$}}
\newcommand{\topref}[2]{\stackrel{\eqref{#1}}#2}
\newcommand{\Leb}[1]{{\mathscr L}^{#1}}      

\newcommand{\down}{\downarrow}              
\newcommand{\up}{\uparrow}
\newcommand{\eps}{\varepsilon}  
\newcommand{\nchi}{{\raise.3ex\hbox{$\chi$}}}

\newcommand{\Rd}{{\R^d}}

 \newenvironment{G}{}{}
 \renewcommand{\GG}[1]{{#1}}
 
 \newcommand{\GGr}[1]{}
\newcommand{\Killed}[1]{}
\renewcommand{\d}{{\mathrm d}}

\renewcommand{\div}{{\nabla\cdot}}
\newcommand{\Cost}[3]{{\mathcal C_{#1}(#2,#3)}}
\newcommand{\phiu}{\phi^1}
\newcommand{\phid}{\phi^2}
\newcommand{\varphiu}{\varphi^1}
\newcommand{\varphid}{\varphi^2}
\newcommand{\varphij}{\varphi^j}
\newcommand{\xu}{x_1}
\newcommand{\xd}{x_2}
\newcommand{\yu}{y_1}
\newcommand{\yd}{y_2}
\newcommand{\byu}{\bar y_1}
\newcommand{\byd}{\bar y_2}
\newcommand{\Oper}[1]{\LL[#1]}
\newcommand{\Opernm}[2]{\LL_{#1}[#2]}
\newcommand{\variabledoubling}{variable-doubling\ }
\renewcommand{\D}{{\mathrm D}}
\newcommand{\Led}{\Leb{d}}      
\let\mu\rho

\newcounter{indice}

\newtheorem{teor}{Theorem}[section]
\newtheorem*{tteor}{Theorem}
\newtheorem{lemma}[teor]{Lemma}
\newtheorem{prop}[teor]{Proposition}
\newtheorem{cor}[teor]{Corollary}

\theoremstyle{definition}
\newtheorem{remark}[teor]{Remark}

\newcommand{\meta}[1]{}

\title{Contraction of general transportation costs along solutions to Fokker-Planck equations with monotone drifts}
\author{Luca Natile\thanks{Dipartimento di Matematica,
 Universit\`a di Pavia, \textsf{luca.natile@unipv.it}}
\and
Mark Peletier\thanks{Department of Mathematics and Computer Science
  and Institute for Complex Molecular Systems,
 Technische Universiteit Eindhoven, \textsf{m.m.peletier@tue.nl}}
\and
 Giuseppe Savar\'e\thanks{Dipartimento di Matematica,
 Universit\`a di Pavia, \textsf{giuseppe.savare@unipv.it}}}

\begin{document}
\maketitle
\begin{abstract}
We shall prove new contraction properties of general transportation costs
along nonnegative measure-valued 
solutions to Fokker-Planck equations in $\Rd$, when the drift is a
monotone (or $\lambda$-monotone) operator.
A new duality approach to contraction estimates
has been developed: it relies on the Kantorovich dual formulation of
optimal transportation problems and on a \variabledoubling  technique. The latter is used to derive a new comparison property of solutions of the backward Kolmogorov (or dual)
equation.
The advantage of this technique is twofold: it directly applies to
distributional solutions without requiring stronger regularity and it
extends the Wasserstein theory of Fokker-Planck equations with gradient drift terms started
by \textsc{Jordan-Kinderlehrer-Otto} \cite{Jordan-Kinderlehrer-Otto98}
to more general costs and monotone drifts, without requiring the drift to be a gradient and without assuming any growth conditions.
\end{abstract}

\section{Introduction}

 \meta{General points: 
 \begin{enumerate}
 \item should we discuss the expected regularity of the solution? For
   instance, if $B$ is discontinuous, what regularity do we get for
   $\rho$? See the marginal comment below.
   \begin{G}
     I added a comment just before Theorem 1; instead of the
     regularization effect in $L^p$ (which, at least locally, always holds by
     [7]), I just observed that even the invariant measure could be
     non-smooth.
   \end{G}
 \item Should we note that the comparison property
   (Theorem~\ref{thm:crucial}) is proved for smooth data, but should
   hold much more generally?
   {I thought that $C^{2,1}_\rmb$ would be a natural assumption for
     applying classical maximum principle; what kind of regularity do
     you have in mind?}\\
     Anything that one can approximate pointwise by solutions of the smooth problem.
     I added a suggestion in the text below.
     \GG{I just removed ``unbounded'' referred to solutions, since I am not sure
       that you can apply the method and prove uniqueness in that
       case. My doubts come from the fact that when you argue by approximation, you get a
       uniqueness result just for the class of \emph{approximable
         solution}, which could be a proper subclass of all the
       reasonable solutions of the equation}
 \item
   \GG{I restored variable-doubling}
 \end{enumerate}}

The aim of this paper is to obtain new uniqueness and contractivity results for
nonnegative measure-valued solutions to the Fokker-Planck equation 
\begin{equation} \label{eq:1}
  \partial_t \rho-\Delta \rho-\div(\rho B)=0,\quad \rho\restr{t=0}=\rho_0,
\end{equation}
where $B:\Rd\to\Rd$ is a Borel $\lambda$-monotone operator, $\lambda\in \R$, i.e.
\begin{equation}
  \label{eq:5}
  \langle B(x)-B(y),x-y\rangle\ge \lambda \,\big|x-y\big|^2\quad\text{for every }x,y\in \Rd.
\end{equation}
Here we consider a weakly continuous family of probability measures
$(\mu_t)_{t\ge0} \subset \PP(\Rd)$ satisfying 
the equation \eqref{eq:1} in the sense of distributions
\begin{equation} \label{eq:2}
  \int_0^{+\infty}\int_\Rd \Big(\partial_t\zeta+\Delta\zeta-B\cdot
  \nabla \zeta\Big)\,\d\mu_t\d t=0\quad\forall\, \zeta\in
  C^\infty_{\mathrm c}(\Rd\times (0,+\infty)),
\end{equation}
with the initial datum $\mu_0$.

Equations of this type are the subject of several papers by
\textsc{Bogachev}, \textsc{Da Prato}, \textsc{Krylov},
\textsc{R\"ockner}, and \textsc{Stannat},
{who consider a very general situation where the Laplacian is replaced by a second order elliptic operator with variable coefficients and $B$ is locally bounded.}
Existence of solutions has been proved by  \cite[Cor.~3.3]{Bogachev-DaPrato-Rockner08}, 
uniqueness has been considered in
\cite{Bogachev-DaPrato-Rockner-Stannat07} under general
growth-coercivity conditions on $B$,
and regularity has been investigated by \cite{Bogachev-Krylov-Rockner01}: in particular, it has been shown that
$\mu_t$ is absolutely continuous {with respect to the Lebesgue measure} for  $\Leb{1}$-a.e.~$t$.

  When $B$ is Lipschitz continuous, uniqueness can be obtained by
  standard duality arguments, see
  e.g.~\cite[Sec.~3]{Ambrosio-Savare-Zambotti09}.
  Here we want to obtain a more precise stability estimate on the
  solutions of \eqref{eq:1}, only assuming
  monotonicity of $B$ without any growth condition.
  To achieve this aim, we adopt the point of view of optimal
  transportation.

  \paragraph{The Wasserstein approach to Fokker-Planck equation in the
    gradient case.}
  When $B$ is the gradient of a $\lambda$-convex function $V:\Rd\to\R$
  then \eqref{eq:1} can be considered as the \emph{gradient flow} of
  the perturbed entropy functional 
  \begin{equation}
    \label{eq:6}
    \HH(\rho):=\int_\Rd u(x)\log u(x)\,\d x+\int_\Rd
    V(x)\,\d\rho(x)\quad
    \rho=u\Leb d
  \end{equation}
  in the space $\PP_2(\Rd)$ of probability measures with finite
  quadratic moments endowed with the so called $L^2$-Kantorovich-Rubinstein-Wasserstein
  distance $W_2(\cdot,\cdot)$. This distance can be defined 
  by
  \begin{equation}
    \label{eq:7}
    \begin{aligned}
      W_2(\rho^1,\rho^2):=\min\Big\{&\int_{\Rd\times\Rd}\big|x_1-x_2\big|^2\,\d\rrho(x_1,x_2):\rrho\in
      \PP(\Rd\times\Rd)\\&
      \rrho\text{ is a coupling between $\rho^1$ and
        $\rho^2$}
      \Big\}      
    \end{aligned}
  \end{equation}
  in terms of couplings, i.e.\ measures $\rrho$ in the product space
  $\Rd\times\Rd$ whose marginals are $\rho^1$ and $\rho^2$
  respectively, so that $\rrho(E\times\Rd)=\rho^1(E)$ and
  $\rrho(\Rd\times E)=\rho^2(E)$ for every Borel subset $E\subset
  \Rd$. It is possible to prove that \emph{optimal couplings}
  realizing the minimum in \eqref{eq:7} always exist.
  
  This remarkable interpretation found in
  \cite{Jordan-Kinderlehrer-Otto98} gave rise to a series of studies on the
  relationships between certain classes of diffusion equations and
  distances between probability measures induced by optimal transport
  problems (see e.g. the general overviews of
  \cite{Villani03,Ambrosio-Gigli-Savare08,Villani09}).
  One of the strengths of this approach is a new geometric insight
  (developed in \cite{Otto01}) in
  the evolution process: in the case of \eqref{eq:1} the
  $\lambda$-convexity of the potential $V$ reflects a
  $\lambda$-convexity property (also called \emph{displacement
    convexity}) of the functional $\HH$ along the geodesics of
  $\PP_2(\Rd)$. This nice feature, discovered by \cite{McCann97},
  suggests that one can adapt some typical basic existence,
  approximation, and regularity results for
  gradient flows of convex functionals in Euclidean spaces or
  Riemannian manifolds
  to the measure-theoretic setting of $\PP_2(\Rd)$. This program has
  been carried out (see e.g.\ \cite{Ambrosio-Gigli-Savare08}) and,
  among the most interesting estimates, it provides the
  $\lambda$-contraction property
  \begin{equation}
    \label{eq:8}
    W_2(\rho^1_t,\rho^2_t)\le \rme^{-\lambda
      t}\,W_2(\rho_0^1,\rho_0^2)\qquad\text{for every }t\ge 0,
  \end{equation}
  where $\rho^i_t$, $i=1,2$, are the solutions to \eqref{eq:1}
  starting from the initial data $\rho_0^i\in \PP_2(\Rd)$.
  \paragraph{Two strategies for the derivation of the
    contraction estimate \eqref{eq:8} in the gradient case.}
  In order to prove \eqref{eq:8} in the gradient case $B=\nabla V$, essentially
  two basic strategies have been proposed:
  \begin{enumerate}
  \item A first approach, developed by
    \cite{Carrillo-McCann-Villani06} for smooth evolutions and by
    \cite{Ambrosio-Gigli-Savare08} in a measure-theoretic setting,
    starts from equation \eqref{eq:1} written in the form
    \begin{equation}
      \label{eq:10}
      \partial_t\rho+\div(\rho\vv)=0,\quad \vv=-\Big(\frac{\nabla
        u}u+\nabla V\Big),\quad \rho=u\,\Leb d,
    \end{equation}
 and it is based on two ingredients: the first one is the formula which evaluates the
  derivative of  the squared Wasserstein distance from a fixed
  measure $\sigma$ along {the (absolutely continuous)} curve $\rho$
  in $\PP_2(\Rd)$
  \begin{equation}
    \label{eq:11}
    \frac\d{\d t}\frac 12
    W_2^2(\rho_t,\sigma)=\int_{\Rd\times\Rd}\langle
    \vv_t(x),y-x\rangle \,\d\rrho_t(x,y)\quad \text{for $\Leb 1$-a.e.\ $t>0$}
  \end{equation}
  where $\rrho_t$ is an optimal coupling between $\rho_t$ and
  $\sigma$.

  The second ingredient is the ``subgradient'' property of the vector
  field $\vv_t$ given by \eqref{eq:10}, related to the displacement
  convexity of $\HH$: in the case
  $\lambda=0$ it  reads as
  \begin{equation}
    \label{eq:12}
    \int_{\Rd\times\Rd}\langle
    \vv_t(x),y-x\rangle \,\d\rrho_t(x,y)\le
    \HH(\sigma)-\HH(\rho_t)\quad
    \text{if }\vv_t=-\Big(\frac{\nabla u_t}{u_t}+\nabla V\Big).
  \end{equation}
  Combination of \eqref{eq:11} and \eqref{eq:12} yields the so called
  Evolution Variational Inequality
  \begin{equation}
    \label{eq:13}
    \frac\d{\d t}\frac 12
    W_2^2(\rho_t,\sigma)\le \HH(\sigma)-\HH(\rho_t)\quad
    \text{for every }\sigma\in \PP_2(\Rd)
  \end{equation}
  which easily yields \eqref{eq:8} for $\lambda=0$ by a \variabledoubling 
  argument (see \cite[Theorem 11.1.4]{Ambrosio-Gigli-Savare08}).

  The main technical point here is that \eqref{eq:12} requires
  $\vv_t\in L^2(\rho_t)$ and \eqref{eq:11} holds
  if
  {for every $0<t_0<t_1<+\infty$ }
  \begin{equation}
    \label{eq:14}
    \int_{t_0}^{t_1} \int_\Rd |\vv_t|^2\,\d\rho_t\,\d t=    \int_{t_0}^{t_1} \int_\Rd
    \left|\frac{\nabla u_t}{u_t}+\nabla V\right|^2\,\d\rho_t\,\d
    t<+\infty,
  \end{equation}
  which should be imposed (in a suitable distributional sense) as an
  \emph{a priori} regularity assumption on the 
  solution of \eqref{eq:1}. {We do not know if solutions to
    \eqref{eq:2} exhibit a similar regularization effect.}
  A second, even more difficult
  point prevents a simple extension of
  \eqref{eq:13} to the general non-gradient case: it is the lack of a
  potential $V$ and therefore of an entropy-like functional $\HH$
  satisfying
  an inequality similar to \eqref{eq:12}.
  \item
    A second approach has been proposed by
    \cite{Otto-Westdickenberg05} and further developed in \cite{Daneri-Savare08,Carrillo-Lisini-Savare-Slepcev10}: it is based on the
    \textsc{Benamou-Brenier} \cite{Benamou-Brenier00} representation formula for the
    Wasserstein distance
    \begin{equation}
      \label{eq:15}
      \begin{aligned}
        &W_2^2(\mu_0,\mu_1)= \inf\Big\{\int_0^1 \int_\Rd
        |\vv_t |^2\,\d \rho_t\,\d t\,:\\
        &\quad
        \partial_t\rho_t+\nabla\cdot(\rho_t\mathbf\vv_t)=0\ \text{in
          $\Rd\times (0,1)$,}\quad \mu_0=\rho\restr{t=0},\quad
        \mu_1=\rho\restr{t=1}\Big\}
      \end{aligned}
    \end{equation}
    and on a careful analysis of the effect of the evolution semigroup
    generated by the equation on curves in $\PP_2(\Rd)$ and its 
    Riemannian tensor $\int_\Rd |\vv|^2\,\d\rho$.
    This technique involves various repeated differentiations and
    works quite well if a nice semigroup preserving smoothness and
    strict positivity of the densities has already
    been defined. Once contraction has been proved on smooth initial
    data, the evolution can be extended to more general ones but
    it seems hard to extend the uniqueness result to cover a general
    distributional solution to the equation.
  \end{enumerate}

  \paragraph{Main result of the paper: contraction estimates for distributional
  solutions.}
  Our purpose is twofold:
  \begin{itemize}\item 
    First of all we want to find a new approach working directly on
    measure-valued solutions to \eqref{eq:1} just satisfying the usual
    distributional formulation \eqref{eq:2}.\\
    \begin{G}
      We note that in general \eqref{eq:1} does not exhibit the same
      regularization effect of the heat equation.  Even in the
      gradient case $B=\nabla V$, there exist solutions $\rho_t$ to
      \eqref{eq:2} which are not of class $C^1(\Rd)$ for every
      $t\ge0$: take, e.g., the invariant measure $\rho_t\equiv
      Z^{-1}\rme^{-V}$ for a suitable convex function $V\not\in
      C^1(\Rd)$ with $\rme^{-V}\in L^1(\Rd)$.  Moreover,
      distributional solutions are easily obtained by approximation
      arguments, as regularization or splitting methods, and they
      should be better suited to deal with the infinite dimensional
      case, as in \cite{Ambrosio-Savare-Zambotti09}: a stability
      result for such a weak class of solutions should be useful in
      these cases.
    \end{G}
    \item Second, we want to cover the case of an arbitrary monotone field
    $B$, without any growth restriction, and to extend contraction
    estimates to more general transportation costs.
  \end{itemize}
  To this aim, let us first introduce the general cost functional
\begin{equation} \label{cost}
\begin{aligned}
  \Cost h{\mu^1}{\mu^2} := \inf\Big\{&\int_{\Rd\times \Rd}h(|x_1-x_2|)\,\d\mmu(x_1,x_2):
  \mmu \in \PP(\Rd\times\Rd), \\&
  \rrho\text{ is a coupling between $\rho^1$ and
        $\rho^2$}\Big\}.
\end{aligned}
\end{equation}
Throughout this paper we assume that
\begin{quote}
$h:[0,+\infty)\to [0,+\infty)$ is a continuous and \textbf{non-decreasing} function with $h(0)=0$.
\end{quote}

  Among the possible interesting choices of $h$, the case
    $h(r):=r^p$ is associated to the family of $L^p$ Wasserstein
    distances (whose $L^2$-version has been introduced in
    \eqref{eq:7})
  on the space $\PP_p(\Rd)$ of all the probability measures with
  moment of order $p$.
  When $h$ is a bounded concave function {satisfying $h(r)>0$ if
    $r>0$, $d(x,y):= h(|x-y|)$
    is a bounded and complete distance function on $\Rd$ inducing the usual
    euclidean topology so that }
  $\Cost h\cdot\cdot$ is a
  complete metric on the space $\PP(\Rd)$ whose topology coincides
  with the usual weak one
  (see, e.g., \cite[Proposition
    7.1.5]{Ambrosio-Gigli-Savare08}).
  
  Since we are not assuming any homogeneity on the general cost
  function $h$, its rescaled versions
  \begin{equation}
    \label{eq:16}
    h_s(r):=h(r\,\rme^s)\quad s\in \R,\ r\ge0
  \end{equation}
  will be useful.
Let us now state our main result:
\begin{teor} \label{teor:main}
  If $\rho^1,\rho^2$ are two distributional solutions to \eqref{eq:2}
  satisfying the sum\-mability condition
  \begin{equation}
    \label{Cond}
    \int_{t_0}^{t_1}\int_\Rd |B(x)-\lambda \, x|\,\d\mu_t(x)\,\d t< +\infty\quad\text{for every }0<t_0<t_1<+\infty,
  \end{equation}
  then they satisfy
  \begin{equation} \label{++}
    \Cost {h_{\lambda t}}{\mu^1_t}{\mu^2_t} \le \Cost
    h{\mu^1_0}{\mu^2_0}\quad\text{for every }t\ge0.
  \end{equation}
  In particular, if $\rho^1_0=\rho^2_0$ then $\rho^1$ and $\rho^2$
  coincide for every time $t\ge0$.
\end{teor}
Let us make explicit some consequences of \eqref{++} according to
the different signs of $\lambda$ and the behaviour of $h$ near $0$ and $+\infty$:
\begin{cor} \label{cor:main}
  Let $\rho^1,\rho^2$ be two distributional solutions to \eqref{eq:2}
  satisfying \eqref{Cond}. 
\begin{enumerate}[a)]
\item If $B$ is monotone, i.e.\ $\lambda \ge 0$, then
  $$\Cost h{\mu^1_t}{\mu^2_t} \le \Cost h{\mu^1_0}{\mu^2_0}.$$
\item If $B$ is $\lambda$-monotone with $\lambda > 0$ and $h$
  satisfies for some exponent $p>0$
  \begin{equation}
    \label{eq:9}
    h(\alpha\, r)\ge \alpha^p h(r)\quad \text{for every $\alpha\ge 1$
    and $r\ge0$}
  \end{equation}
  then
  $$\Cost h{\mu^1_t}{\mu^2_t} \le \rme^{-p\lambda \,t}\,\Cost h{\mu^1_0}{\mu^2_0}.$$
\item If $B$ is $\lambda$-monotone with $\lambda < 0$ and $h$
  satisfies
  for some exponent $p>0$
  \begin{displaymath}
    h(\alpha\, r)\ge \alpha^p h(r)\quad \text{for every $\alpha\le 1$
    and $r\ge0$}
  \end{displaymath}
  then
  $$\Cost h{\mu^1_t}{\mu^2_t} \le \rme^{-p\lambda \,t}\,\Cost
  h{\mu^1_0}{\mu^2_0}.$$
\end{enumerate} 
In the particular case of the Wasserstein distance $W_p$, $p\ge1$, we have
\begin{equation} \label{8} 
  W_p (\mu^1_t,\mu^2_t ) \le \rme^{-\lambda t}\,  W_p (\mu^1_0,\mu^2_0 ).
\end{equation}
\end{cor}
Theorem \ref{teor:main} has a simple application to invariant
measures $\rho_\infty\in \PP(\Rd)$, which are stationary solutions of
\eqref{eq:2} and therefore satisfy
\begin{equation}
  \label{eq:66}
  \int_\Rd \Big(\Delta\zeta-B\cdot
  \nabla \zeta\Big)\,\d\mu_\infty=0\quad\forall\, \zeta\in
  C^\infty_{\mathrm c}(\Rd).
\end{equation}
\begin{cor}[Strongly monotone operators and invariant measures]
  Let us suppose that $B$ is strongly monotone, i.e.\
  $\lambda>0$. Then equation \eqref{eq:66}
  has at most one solution $\rho_\infty\in \PP(\Rd)$ satisfying the
  integrability condition
  \begin{equation}
    \label{eq:67}
    \int_\Rd |Bx-\lambda x|\,\d\rho_\infty(x)<\infty.
  \end{equation}
  For each solution $\rho_t$ to \eqref{eq:2}-\eqref{Cond} and each
  cost $h$ satisfying \eqref{eq:9} we have
  \begin{equation}
    \label{eq:68}
    \Cost h{\rho_t}{\rho_\infty}\le \rme^{-p\lambda \,(t-t_0)}\Cost h{\rho_{t_0}}{\rho_\infty}.
  \end{equation}
\end{cor}
Note that in the case $\lambda>0$ condition \eqref{eq:67} is weaker
than
$B\in L^1(\rho_\infty;\Rd)$.
\begin{remark}[An equivalent formulation of the contraction estimate]
  \label{rem:equivalent}
  We can give an equivalent version of \eqref{++} by keeping fixed the
  cost but rescaling the measures. In fact, we can associate to the
  solutions $\rho^1,\rho^2$ of \eqref{eq:2} their rescaled versions
  $\tilde\rho^1,\tilde\rho^2$ defined by
  \begin{equation}
    \label{eq:50}
    \tilde\rho^j(E):=\rho^j(\rme^{-\lambda t} \,E)\quad\text{for every
      Borel set }E\subset \Rd,\ j=1,2.
  \end{equation}
  Then $\tilde\rho^j$ is the push-forward of $\rho^j$ through the map
  $x\mapsto \rme^{\lambda t}x$ and satisfies the change-of-variables
  formula
  \begin{equation}
    \label{eq:51}
    \int_\Rd \zeta(y)\,\d\tilde\rho^j(y)=\int_\Rd \zeta(\rme^{\lambda
      t}\,x)\,\d\rho^j(x)\quad
    \text{for every }\zeta\in C_\rmb(\Rd).
  \end{equation}
  Inequality \eqref{++} is then equivalent to
  \begin{equation}
    \label{eq:52}
    \Cost h{\tilde\rho^1_t}{\tilde\rho^2_t}\le \Cost
    h{\rho^1_0}{\rho^2_0}\quad
    \text{for every }t>0.
  \end{equation}
\end{remark}
\paragraph{Strategy of the proof: Kantorovich duality and a
  \variabledoubling  technique.}
In order to prove Theorem \ref{teor:main} we develop a new strategy,
generalizing
\cite{Portegies-Peletier08}.
It relies on the well-known dual Kantorovich formulation~\cite{Villani03} of
the transportation cost \eqref{cost}:
\begin{equation} \label{dualintro}
  \begin{aligned}
    \Cost h{\mu^1}{\mu^2}=\sup\Big\{& \int_\Rd \phiu\,\d\mu^1+
    \int_\Rd \phid\,\d\mu^2:\\&\phiu,\phid\in C_{\rm b}(\Rd),\
    \phiu(\xu)+\phid(\xd)\le h(|\xu-\xd|)\Big\}.
  \end{aligned}
\end{equation}
This formula reduces the estimate of the cost $\Cost h{\rho^1_T}{\rho^2_T}$ of two
solutions of \eqref{eq:1} at a certain final time $T$ to the estimate
of
\begin{equation}
  \label{eq:18}
  \Sigma(\phiu,\phid;T):=\int_\Rd \phiu\,\d\mu^1_T+
    \int_\Rd \phid\,\d\mu^2_T
\end{equation}
for an arbitrary pair of functions $\phiu,\phid$ satisfying the
constraint
\begin{equation}
  \label{eq:17}
  \phiu(\xu)+\phid(\xd)\le h(|\xu-\xd|)\quad\text{for every
  }\xu,\xd\in \Rd.
\end{equation}
Assuming for the sake of simplicity that $B$ is monotone, bounded
and smooth, we can obtain an estimate of $\Sigma(\phiu,\phid;T)$ by
solving the final-value problem for the adjoint equation
\begin{equation}
  \label{eq:19}
  \partial_t \phi^i+\Delta\phi^i-B\cdot\nabla\phi^i=0\quad\text{in
  }\Rd\times(0,T),\quad
  \phi^i(\cdot,T):=\phi^i
\end{equation}
since the distributional formulation \eqref{eq:2} yields
\begin{equation}
  \label{eq:20}
  \Sigma(\phiu_T,\phid_T;T)=\Sigma(\phiu_0,
  \phid_0;0)
\end{equation}
The following crucial result, based on a
``\variabledoubling  technique'', provides the final step, showing that
$\phiu_0,\phid_0$ still satisfy the constraint
\eqref{eq:17} so that $\Sigma(\phiu_0,
\phid_0;0)\le \Cost h{\rho^1_0}{\rho^2_0}$.
\begin{teor}
  \label{thm:crucialintro}
If $\phiu,\phid\in C^{2,1}_{\mathrm b}(\Rd\times [0,T])$ are solutions
of \eqref{eq:19} in the case when $B$ is monotone, bounded and smooth,
such that
\begin{equation*}
  \phiu(\xu,T)+\phid(\xd,T)\le h(|\xu-\xd|)\quad\forall\,\xu,\xd\in \Rd,
\end{equation*}
then
\begin{equation*}
  \phiu(\xu,0)+\phid(\xd,0)\le h(|\xu-\xd|)\quad\forall\,\xu,\xd\in \Rd.
\end{equation*}
\end{teor}

\begin{remark}
While we prove Theorem~\ref{thm:crucialintro} for bounded and smooth
drifts $B$, and solutions $\phi^{1,2}\in C^{2,1}_{\mathrm b}(\Rd\times
[0,T])$, the property clearly carries over to any pointwise limit of
such solutions. We therefore expect it to hold for a much larger class
of monotone drifts $B$ and \GGr{for unbounded} solutions.
\end{remark}



\paragraph{Plan of the paper}

In section 2, we collect some tools useful to our arguments: we
present a slightly refined version of Kantorovich duality, an
approximation technique of the cost functional, the construction of a
smooth and bounded approximation of the operator $B$, and a rescaling
trick which allows to consider $\lambda=0$ in the following arguments.
Section 3 is devoted to the proof of Theorem \ref{thm:crucialintro},
the last Section contains the proof of Theorem \ref{teor:main}.  

\section{Preliminaries} 

  In this section we collect some preliminary and technical
  regularization results which will turn to be useful in the sequel.

\subsection{{$C^\infty_{\rmc}(\Rd)$ functions in} Kantorovich duality}

  Let us first show that we can assume $\phiu,\phid$ are smooth and
  compactly supported in the duality formula \eqref{dualintro}.

\begin{prop}\label{Kant+}
If the cost function $h$ is Lipschitz continuous and satisfies 
$  \lim_{r\uparrow+\infty}h(r)=+\infty,$
then
\begin{equation} \label{Kdual[II]}
  \begin{aligned}
    \Cost h{\mu^1}{\mu^2}=\sup\Big\{& \int_\Rd \phiu\,\d\mu^1+
    \int_\Rd \phid\,\d\mu^2:\\&\phiu,\phid\in C^\infty_{\rmc}(\Rd),\
    \phiu(\xu)+\phid(\xd)\le h(|\xu-\xd|)\Big\}.
  \end{aligned}
\end{equation}
\end{prop} 
\begin{proof}
  
    Let us recall that the $h$-transform of a given bounded function
    $\zeta:\Rd\to\R$ is defined as
    \begin{equation}
      \label{eq:22}
      \zeta^h(x):=\inf_{y\in \Rd}h(|x-y|)-\zeta(y),
    \end{equation}
    and it is a bounded and Lipschitz continuous function satisfying $\zeta(x)+\zeta^h(y)\le h(|x-y|)$.
    
    Let us fix $c<\Cost h {\mu^1}{\mu^2}$ and admissible $\phiu,\phid\in C_{\rmb}(\Rd)$ such that
    \begin{equation}
      \label{eq:24}
      \int_\Rd \phiu\,\d\mu^1+
    \int_\Rd \phid\,\d\mu^2>c.
    \end{equation}
    By possibly replacing $\phid$ with $(\phiu)^{h}\ge \phid$
    and $\phiu$ with $(\phiu)^{hh}\ge \phiu$, it is not restrictive to assume that
    $\phiu,\phid$ are also Lipschitz continuous. Adding to $\phiu$ and subtracting from $\phid$
    a suitable constant, we can also assume that $\phiu\ge0$ and $\phid\le 0$.

Let us now consider a  family of mollifiers

  $\kappa_\eta$ and of cutoff functions $\nchi_R$ defined by
  \begin{subequations}
    \begin{gather}
      \label{eq:32}
      \kappa_\eta(x):=\eta^{-d}\kappa\big(x/\eta\big),\quad
      \nchi_R(x):=\nchi(x/R)\quad x\in\Rd,\
      \eta, R>0,
      \intertext{where $\kappa,\nchi\in C^\infty_{\rmc}(\Rd)$ satisfy}
      \label{eq:69}
      \kappa\ge 0,\quad \int_{\Rd} \kappa(x)\, \d x =1,\qquad
      0\le \nchi\le1,
      \quad
      \nchi(x)=0\text{ if }|x|\ge1,
      \quad\nchi(x)=1\text{ if }|x|\le 1/2.
    \end{gather}
  \end{subequations}

 \meta{Symmetry and decreasing property of $\kappa$ are relevant just in dimension 1\\
   and could be confusing at this point since they are not required in the lemma.\\
   I separated
   the role of the mollifiers and of the cut-off\\ and
   I re-stated the similar assumptions later on\marginpar{I fully agree!}}
 
  \bigskip
We set $\phiu_\eta:=\phiu\ast \kappa_\eta$ and 
$\phid_\eta:=\phid\ast \kappa_\eta-\delta_\eta$, where
\begin{displaymath}
 \delta_\eta:=\sup(\phiu\ast\kappa_\eta-\phiu)^+
 +\sup(\phid\ast\kappa_\eta -\phid)^+.
  \end{displaymath}
  The definition of $\delta_\eta$ yields
  \begin{displaymath}
    \phiu_\eta(\xu)+\phid_\eta(\xd)\le 
    \phiu\ast\kappa_\eta(\xu)-\phiu(\xu)+\phid\ast\kappa_\eta(\xd)-\phid(\xd)-\delta_\eta+h(|\xu-\xd|)
    \le h(|\xu-\xd|).
  \end{displaymath}
  Moreover, since $\phiu,\phid$ are Lipschitz, $\phiu_\eta$ and $\phid_\eta$ converge to $\phiu,\phid$
  uniformly as $\eta\downarrow0$, so that
  $\phiu_\eta$ and $\phid_\eta$ are a smooth admissible {pair} \Killed{couple} still satisfying \eqref{eq:24}
  and the sign condition $\phiu_\eta\ge0,\ \phid_\eta\le 0$.
  
  Let us now choose $R_0>0$ such that
  \begin{equation}
    \label{eq:25}
    h(r)\ge \sup \phiu_\eta\quad\text{for every $r\ge R_0$}
  \end{equation}
  {Setting $\phiu_{\eta,R}:=\phiu_\eta \, \nchi_R\le \phiu_\eta$} we easily have for $R\ge R_0$
  \begin{displaymath}
    \inf_{\xu\in \Rd} h(|\xu-\xd|)-\phiu_{\eta,R}(\xu)\ge 0\quad\text{if }|\xd|\ge 2R\ge R+R_0.
  \end{displaymath}
    Since $\phid_{\eta,4R}:=\phid_\eta\,\nchi_{4R}$ satisfies
    $\phid_{\eta,4R}(x_2)=\phid_\eta(x_2)$ if $|x_2|\le 2R$
    and $\phid_{\eta,4R}(x_2)\le 0$ for every $x_2\in \Rd,$
  
  it follows that $\phiu_{\eta,R},\phid_{\eta,4R}$
  is an admissible couple in
  $C^\infty_\rmc(\Rd)$, and,
  for $R$ sufficiently large, it still satisfies \eqref{eq:24}.
%
\end{proof}
\subsection{Regularization of the cost function.}
\label{subsec:regularization_cost}
In this section we shall show that it is sufficient to consider
nonnegative, Lipschitz, and unbounded costs
(as those considered in Proposition \ref{Kant+}) 
in the proof of Theorem \ref{teor:main}.

\begin{lemma}
  \label{le:cost_approximation}
  If \eqref{++} holds for every nonnegative Lipschitz and
  nondecreasing cost function $h$
  with $\lim_{r\up+\infty}h(r)=+\infty$, then it holds for every 
  continuous and nondecreasing cost $h$.
\end{lemma}

\begin{proof}
  We first prove that it is sufficient to consider nonnegative
  Lipschitz costs; in a second step, we deal with the asymptotic
  requirement.

  \noindent\underline{Step 1:} \emph{$h$ Lipschitz.}
Adding a suitable constant we can assume that $h(r)\ge h(0)=0$.
We can then approximate $h$ from below by the increasing sequence of nonnegative Lipschitz functions
\begin{equation*}
  h^n(r):=\inf_{s\ge 0} h(s)+n|r-s|
\end{equation*} 
which satisfies
\begin{equation*}
  0=h^n(0)\le h^n(r)\le h(r),\quad
  \lim_{n\uparrow+\infty} h^n(r)=h(r)\quad
  \forall\, r\ge0,
\end{equation*}
the convergence being uniform on each compact interval of $[0,+\infty)$.
Applying Lemma \ref{le:stability} below we find
\begin{displaymath}
  \Cost{h_{\lambda\,t}}{\rho^1_t}{\rho^2_t}\topref{eq:57}=
  \lim_{n\up+\infty}  \Cost{h^n_{\lambda\,t}}{\rho^1_t}{\rho^2_t}
  \topref{++}\le  \liminf_{n\up+\infty}\Cost{h^n}{\rho^1_0}{\rho^2_0}\topref{eq:57}=
  \Cost{h}{\rho^1_0}{\rho^2_0}.
\end{displaymath}
\noindent\underline{Step 2:} \emph{$\lim_{r\up+\infty} h(r)=+\infty$.}
{Let us set $\rho_0:=\rho^1_0+\rho^2_0$,}
let us introduce the function
\begin{displaymath}
  m(r):=\rho_0(\Rd\setminus r\, U),\quad
  U:=\big\{x\in \Rd:|x|< 1\big\},
\end{displaymath}
and let us consider a sequence $r_n$ in $[0,+\infty)$ such that 
\begin{displaymath}
  r_0:=0, \quad
  r_1:=1,\quad r_{n+1}-r_n\ge r_n-r_{n-1}\quad\text{and }m(r_{n+1})\le 2^{-n}.
\end{displaymath}
It is easy to check that $r_n$ is a diverging increasing sequence; if
$g$ is the piecewise linear function satisfying $g(r_n)=n$, i.e.
\begin{displaymath}
  g(r):= n+\frac{r-r_n}{r_{n+1}-r_n}\quad \text{if }r\in [r_n,r_{n+1}],
\end{displaymath}
then $g$ is Lipschitz continuous, increasing, unbounded, concave, and
it satisfies {$g(0)=0$} and 
\begin{align*}
  G:=&\int_\Rd g(|x|)\,\d\rho_0(x)=
    \int_\Rd\Big(\int_0^{|x|}g'(r)\,\d r\Big) \,\d\rho_0(x)=
    \int_\Rd\Big(\int_0^{+\infty} g'(r)\mathds 1_{r\le |x|} \,\d r\Big) \,\d\rho_0(x)  
  \\=&
  \int_0^{\infty} g'(r)\, m(r)\,\d r=\sum_{n=1}^{+\infty}\frac{1}{
    r_n-r_{n-1}}\int_{r_{n-1}}^{r_{n}}m(r)\,\d r\le \sum_{n=0}^{+\infty} m(r_n)<+\infty.
\end{align*}
We can thus consider the perturbed cost
\begin{equation*}
  h^{\eps}(r):=h(r)+\eps\,  g(r)
\end{equation*} 
which is Lipschitz, increasing, unbounded. Since $g$ is concave,
increasing, {and $g(0)=0$}, we
have
\begin{equation}
  \label{eq:54}
  g(|\xu-\xd|)\le g(|\xu|+|\xd|)\le g(|\xu|)+g(|\xd|)\quad\text{for
    every }\xu,\xd\in \Rd,
\end{equation}
so that if {$\rrho_0$} \Killed{$\rrho$}
is an optimal coupling between {$\rho^1_0$ and $\rho^2_0$}
\Killed{$\rho^1$ and $\rho^2$}
for the cost $h$ (we can assume that the initial cost is finite), then
\begin{align*}
  \Cost h{\rho^1_0}{\rho^2_0}&\le \Cost{h^\eps}{\rho^1_0}{\rho^2_0}\le
  \Cost h{\rho^1_0}{\rho^2_0}+\eps\int_{\Rd\times\Rd}
  g(|\xu-\xd|)\,\d\rrho_0(\xu,\xd)\\&\topref{eq:54}\le
  \Cost h{\rho^1_0}{\rho^2_0}+\eps\int_{\Rd\times\Rd}
  \Big(g(|\xu|)+g(|\xd|)\Big)\,\d\rrho_0(\xu,\xd)=
  \Cost h{\rho^1_0}{\rho^2_0}+\eps G.  
\end{align*}
Therefore, if Theorem \ref{teor:main} holds for $h^\eps$ we have
\begin{displaymath}
  \Cost h{\rho^1_t}{\rho^2_t}\le
  \Cost {h^\eps}{\rho^1_t}{\rho^2_t}\le
  \Cost {h^\eps}{\rho^1_0}{\rho^2_0}\le
 \Cost h{\rho^1_0}{\rho^2_0}+\eps G.  
\end{displaymath}
Passing to the limit as $\eps\down0$ we conclude.
\end{proof}
The following result provides a variant of well known stability properties
of transportation costs (see \cite[Theorem 3]{Schachermayer-Teichmann09}, \cite[Theorem 5.20]{Villani09})
and holds the much more general setting of optimal transportation in Radon metric spaces
\cite[Chapter 6]{Ambrosio-Gigli-Savare08}.
\begin{lemma}[Lower semicontinuity of the cost functional w.r.t.\ local uniform convergence of $h$]
  \label{le:stability} \ \\
  Let $h:[0,+\infty)\to [0,+\infty)$ be a continuous cost function and let $h^n:[0,+\infty)\to [0,+\infty)$
  be a sequence of lower semicontinuous functions converging to $h$ locally uniformly in $[0,+\infty)$.
  For every couple $\rho^1,\rho^2\in \PP(\Rd)$ we have
  \begin{equation}
    \label{eq:56}
    \liminf_{n\up+\infty}\Cost{h^n}{\rho^1}{\rho^2}\ge \Cost{h}{\rho^1}{\rho^2}.
  \end{equation}
  In particular, if $h^n\le h$ for every $n\in \N$ then 
  \begin{equation}
    \label{eq:57}
    \lim_{n\to+\infty}\Cost{h^n}{\rho^1}{\rho^2}=\Cost h{\rho^1}{\rho^2}.
  \end{equation}
\end{lemma}
\begin{proof}
  Let us set $H^n(\xu,\xd):=h^n(|\xu-\xd|)$ and observe that $H^n$ converges to
  $H(\xu,\xd):=h^n(|\xu-\xd|)$ uniformly on compact sets of $\Rd\times\Rd$.
  If $\rrho_n\in \PP(\Rd\times\Rd)$ is an optimal coupling between $\rho^1,\rho^2$
  with respect to the cost $h^n$ then
  \begin{displaymath}
    \Cost{h^n}{\rho^1}{\rho^2}=
    \int_{[0,+\infty)} z\,\d \mu_n(z),\quad\text{where } \mu_n=(H^n)_\# \rrho_n.
  \end{displaymath}
  Since the marginals of $\rrho_n$ are fixed, the sequence $(\rrho_n)_{n\in \N}$ is tight and up to the
  extraction of a suitable subsequence (still denoted by $\rrho_n$)
  we can suppose that $\rrho_n$ converge to 
  to some limit coupling $\rrho$ between $\rho^1,\rho^2$ in $\PP(\Rd\times\Rd)$.
  Since $\mu_n$ weakly converge to $\mu=H_\# \rrho$ by
  \cite[Lemma 5.2.1]{Ambrosio-Gigli-Savare08}, standard lower semicontinuity
  of integrals with nonnegative continuous integrands \cite[Lemma 5.1.7]{Ambrosio-Gigli-Savare08}
  yields
  \begin{displaymath}
    \liminf_{n\to+\infty}  \int_{[0,+\infty)} z\,\d \mu_n(z)\ge
    \int_{[0,+\infty)} z\,\d \mu(z)=\int_{\Rd\times\Rd}H(\xu,\xd)\,\d\rrho(\xu,\xd)\ge \Cost h{\rho^1}{\rho^2}.
  \end{displaymath}
\end{proof}
\subsection{Bounded, smooth approximations of a monotone operator}
\label{subsec:Bregular}
If $A:\Rd\to\Rd$ is a monotone operator then there exists \cite[Corollary 2.1]{Brezis73} a
maximal monotone {multivalued} extension
$\sfA:\Rd\rightrightarrows{\Rd}$ {(thus taking values in $2^{\Rd}$)}
such that  $A(x)\in
\sfA(x)$ for every $x\in \Rd$. We denote by $\sfA^\circ(x)$ the element of
minimal norm in (the closed convex set) $\sfA(x)$.
\Killed{Applying} \cite[Corollary 1.4]{Alberti-Ambrosio99}
{shows that the set $\sfA(x)\subset \Rd$ reduces to the singleton
  $\{A(x)\}$
  $\Leb d$-almost everywhere}: in fact it satisfies
\begin{equation}
  \label{eq:26}
  \sfA(x)=\{\sfA^\circ(x)\}=\{A(x)\}\quad\text{for $\Led$-a.e.\ $x\in \Rd$},\quad
  \sfA(x)=\mathrm{conv}\big\{\lim_{n\to\infty}A(x_n)\ \text{for
    some }x_n\to x\big\}
\end{equation}
We recall the following important approximation result 
\cite[Theorem 4.1]{Fitzpatrick-Phelps92}:
we denote by $U$ the open unit ball in $\Rd$.
\begin{tteor}[Fitzpatrick-Phelps] \label{Fitz}
For every maximal monotone operator $\sfA:\Rd\rightrightarrows\Rd$, there
exists a sequence of maximal monotone operators
$\sfA_n:\Rd\rightrightarrows\Rd$ such that,
for each $x \in \Rd$ and all n,
\begin{equation} \label{Fitz:eq}
  \sfA(x)\cap n\, U \subset \sfA_n(x)\subset n\,\overline U, 
  \quad \sfA_n(x) \setminus \sfA(x) \subset n  \,\partial U\quad\text{for
    every }x\in \Rd.
\end{equation}
\end{tteor}
Notice that \eqref{Fitz:eq} yields in particular
\begin{equation}
  \label{eq:27}
  |\sfA_n^\circ(x)|= \min\big(|\sfA^\circ(x)|,n\big)\quad\text{for every }x\in \Rd.
\end{equation}
\begin{teor} \label{approxB}
  Let $\sfA:\Rd\rightrightarrows\Rd$ be a maximal monotone operator
  and
  $(\beta_n)_{n\in \N}$ a vanishing sequence of positive real numbers.
  There exists a sequence of smooth, globally Lipschitz, and bounded monotone operators
  $A_n:\Rd\to\Rd$ such that 
\begin{equation} \label{B approx}
  \mathrm{Lip}(A_n)\le n,\quad
  |A_n(x)|\le \min\big(|\sfA^\circ(x)|,n\big)+\beta_n,\quad
  \lim_{n\to+\infty}A_n(x)=\sfA^\circ(x)\quad
  \mbox{for every } x\in \Rd.
\end{equation}
\end{teor}
\begin{proof}
Let $\sfA_n$ be a sequence of maximal monotone operators satisfying
\eqref{Fitz:eq}
and let $Y_n:\Rd\to\Rd$ be the Moreau-Yosida approximation of $\sfA_n$ of
parameter $n^{-1}$ \cite[Proposition 2.6]{Brezis73}
\begin{displaymath}
  Y_n(x):=n\Big(x-(I+n^{-1}\,\sfA_n)^{-1} x\Big)
\end{displaymath}
Note that $Y_n$ is a $n$-Lipschitz monotone map satisfying
\begin{equation}
  \label{eq:28}
  |Y_n(x)|\le
  |\sfA_n^\circ(x)|\topref{eq:27}=\min\big(|\sfA^\circ(x)|,n\big)\quad\text{for every
  }x\in \Rd
\end{equation}
Let us fix $x\in \Rd$ and let $x_n\in \Rd$ be the unique solution of
\begin{equation}
  \label{eq:29}
  x_n+n^{-1}\,\sfA_n(x_n)\ni x\quad\text{so that}\quad
  Y_n(x)=n(x-x_n)\in \sfA_n(x_n).
\end{equation}
If $n>|\sfA^\circ(x)|$ then \eqref{eq:28} yields $Y_n(x)\notin n\,\partial
U$;
applying \eqref{Fitz:eq} and \eqref{eq:28} again we get
\begin{equation}
  \label{eq:30}
  Y_n(x)\in \sfA(x_n),\quad
  |Y_n(x)|\le |\sfA^\circ(x)|,\quad
  |x-x_n|\le n^{-1}\,|\sfA^\circ(x)|\quad\text{for every }n>|\sfA^\circ(x)|.
\end{equation}
Since the graph of $\sfA$ is closed, any accumulation point $y$ of the
bounded sequence $Y_n(x)$ satisfies
\begin{equation}
  \label{eq:31}
  y\in \sfA(x),\quad |y|\le |\sfA^\circ(x)|.
\end{equation}
We thus conclude that $\lim_{n\up+\infty}Y_n(x)=\sfA^\circ(x)$ for
every $x\in \Rd$.

To conclude the proof we need to regularize $Y_n$: to this aim we
consider the family of mollifiers $\kappa_\eta$ as in \eqref{eq:32}
and we set
\begin{equation}
  \label{eq:33}
  A_n:=Y_n\ast \kappa_\eta \quad\text{with}\quad
  \eta:= (n\, k)^{-1}\beta_n \text{ where } k:=\int_\Rd
  |x|\kappa(x)\,\d x,
\end{equation}
so that
\begin{displaymath}
  |A_n(x)-Y_n(x)|\le \eta\,k\, \mathrm{Lip}(Y_n)\le n\,\eta\,k\le \beta_n.
\end{displaymath}
\end{proof}

We consider now a radial smoothing:
\begin{prop}
  \label{approxBn}
  Let $A_n:\Rd\to\Rd$ be smooth, Lipschitz, and bounded monotone
  operators satisfying
  \eqref{B approx}.
  For every $m\in\N$ there exists bounded, smooth,
  Lipschitz, and monotone operators $A_{n,m}$ such that
 \begin{gather}
    \label{eq:34}
    \mathrm{Lip}(A_{n,m})\le n,\quad
    \sup_{x\in \Rd}|A_{n,m}(x)|\le n+\beta_n,\quad
    \sup_{x\in\Rd} \big|\D\,A_{n,m}(x)\cdot x\big|\le 2m\,(n+\beta_n)\\
    \label{eq:62}
    \lim_{m\up+\infty}A_{n,m}(x)=A_n(x)\quad\text{for every }x\in \Rd.
  \end{gather}
\end{prop}
%
%
%
%
\begin{proof}
    We consider a family of mollifiers $\kappa_\eta=\eta^{-1}\kappa(\cdot/\eta)\in C^\infty_\rmc(\R)$,
    where $\kappa$ satisfies
    \begin{equation}
      \label{eq:71}
      \supp(\kappa)\subset [0,2],\quad 0\le \kappa\le \kappa(1)=1,\quad
      (1-x)\kappa'(x)\ge 0,\quad \int_\R\kappa(x)\,\d x=1,
    \end{equation}
  
  and the function $\vartheta\in C^\infty_{\rm
    c}(0,+\infty)$ defined by $\vartheta(r):=\kappa(-\log r)$, $r>0$.
 \Killed{where $\kappa$ is the mollifier of \eqref{eq:32} in dimension $d=1$.}
 We set
  \begin{equation}
    \label{eq:35}
    A_{n,m}(x):=m\int_0^{+\infty}
    A(rx)\vartheta(r^{m})\,\frac{\d r}r
  \end{equation}
  The change of variable $r=e^{-z}$ shows that
  \begin{displaymath}
    A_{n,m}(x)= 
    m\int_\R A_n(x\,\rme^{-z})\,\kappa(m\,z)\,\d z
    =A_n^x\ast \kappa_{1/m} (0),\quad\text{where }
    A_n^x(z):=A_n(x\,e^{z})\ \text{for } z\in \R.
  \end{displaymath}
  It is then easy to check that \Killed{ if $A_n$ has bounded derivatives then}
  {$|\D A_{n,m}|\le n$} \Killed{enjoys the same property}
    since
    \begin{displaymath}
      |D A_{n,m}(x)|\le
      m\int_\R \big|\D A_n(x\,\rme^{-z})\big|\rme^{-z}\,\kappa(m\,z)\,\d z
      \topref{B approx}\le
      n\int_\R \rme^{-y/m}\,\kappa(y)\,\d y
      \topref{eq:71}\le n,
    \end{displaymath}
  
  and {$A_{n,m}$} \Killed{it} converges pointwise to $A_n$ as
  $m\up+\infty$.

  Concerning the second bound of \eqref{eq:34} we easily have
  \begin{align*}
    \D\,A_{n,m}(x)\cdot x &=
    m\int_0^{+\infty}
    \D\,A_n(rx)\cdot x\vartheta(r^{m})\,{\d r}=
    m\int_0^{+\infty}
    \frac\d{\d r}\Big(A_n(rx)\Big)\vartheta(r^{m})\,{\d r}\\&=
    -m^2 \int_0^{+\infty}
    A_n(rx)\tilde \vartheta(r^{m})\,\frac{\d r}r\qquad \text{where }\tilde\vartheta(r):=r\vartheta'(r),
  \end{align*}
  so that the inequality follows choosing $\kappa$ even and
  nondecreasing in $[0,+\infty$), so that $\int_0^{+\infty}
  |\vartheta'(r)|\,\d r=2$.  
\end{proof}
\subsection{$\lambda$-monotonicity and rescaling}
\label{subsec:rescaling}
We show here a simple rescaling argument (inspired by
\cite{Carrillo-Toscani00}, where the rescaling technique has been
applied to a wide class of diffusion equations), which is useful to
deduce the estimates in the general $\lambda$-monotone case to
the simpler case of a monotone operator.

We therefore assume that $\lambda\neq 0$,
and we introduce the time rescaling functions
\begin{gather}
  \label{eq:36}
  \sfs(t):= \int_0^t \rme^{2\lambda r}\,\d r=\frac1{2\lambda}\big(\rme^{2\lambda t}-1\big),\quad
  \sft(s):= \frac1{2\lambda} \log(1+2\lambda s)\quad s\in [0,S_\infty)
  \intertext{where}
  \label{eq:61}
  S_\infty:=
  \begin{cases}
    +\infty&\text{ if $\lambda>0$},\\
    -1/(2\lambda)&\text{ if $\lambda<0$.}
  \end{cases}
\end{gather}

We associate to a family of probability measures $\rho_t$, $t\in [0,T]$,
their rescaled versions $\sigma_s$, $s\in [0,S_\infty)$, defined by
\begin{equation}
  \label{eq:40}
  \sigma_s(E):=\rho_{\sft(s)}(\rme^{-\lambda \sft(s)}\,E)\quad\text{for every Borel set }E\subset \Rd.
\end{equation}
If $B:\Rd\to\Rd$ is a $\lambda$-monotone Borel map we set $A:=B-\lambda
I$ and 
\begin{equation}
  \label{eq:41}
  \tilde B(y,s):=\rme^{-\lambda \sft(s)} B(\rme^{-\lambda \sft(s)}\,y),\quad
  \tilde A(y,s)=
  \rme^{-\lambda \sft(s)} A(\rme^{-\lambda \sft(s)}\,y)
  \quad\text{for }y\in \Rd,\ s\in \R.
\end{equation}
Notice that if $B$ is $\lambda$-monotone, then $A$ and $\tilde
A(\cdot,s)$, $s\in [0,S_\infty)$,  are monotone.
\begin{prop}
  \label{prop:rescaled}
  A continuous family $\rho_t\in \PP(\Rd)$ is a distributional
  solution of \eqref{eq:2} if and only if the rescaled measures
  $\sigma_s$ defined by \eqref{eq:40} and \eqref{eq:36} satisfy
  \begin{equation}
    \label{eq:39}
    \int_{0}^{S_\infty}\int_\Rd \Big(\partial_s\varphi+\Delta\varphi-\tilde A(\cdot,s)\cdot
    \nabla \varphi\Big)\,\d\sigma_s\,\d s=0\quad\forall\, \varphi\in
    C^\infty_{\mathrm c}(\Rd\times (0,S_\infty)).
  \end{equation}
  If $\rho$ satisfies \eqref{Cond} then
  \begin{equation}
    \label{eq:55}
    \int_{s_0}^{s_1}\int_\Rd |\tilde A(x,s)|\,\d\sigma_s\,\d s<+\infty\quad\text{for every }0<s_0<s_1<S_\infty,
  \end{equation}
  and in this case $\sigma$ satisfies
  \begin{equation} \label{eq:3}
    \int_\Rd \varphi(\cdot,s_1)\,\d\sigma_{s_1}-\int_\Rd \varphi(\cdot,s_0)\,\d\sigma_{s_0} =
    \int_{s_0}^{s_1} \int_\Rd \Big(\partial_s\varphi+ \Delta\varphi-
    \tilde A(y,s)\cdot \nabla \varphi\Big)\,\d\mu_s\,\d s.
\end{equation}
for every test function $  \varphi\in C^{2,1}_{\mathrm
  b}(\Rd\times[s_0,s_1])$ 
with bounded first and second derivatives.
\end{prop}
  \begin{proof}
    We introduce the change of variable map 
    $\sfX(x,t):=(\rme^{\lambda t}x,\sfs(t))$ and for a given smooth function
    $\varphi\in C^\infty_\rmc(\Rd\times(0,s_\infty))$ we set 
    $\zeta(x,t):=\varphi(\rme^{\lambda \sfs(t)},\sfs(t))=\varphi\circ \sfX$.
    Denoting by $(y,s)\in \Rd\times[0,s_\infty )$ the new variables, easy calculations show that
    in $\Rd\times (0,+\infty)$ we have
    \begin{align*}
      \partial_t\zeta&=\sfs'\Big(\partial_s\varphi+\lambda \rme^{-2 \lambda t}\nabla_y\varphi\cdot y\Big)\circ \sfX,\quad
      &
      \nabla_x\zeta&=\rme^{\lambda t}\,\nabla_y\varphi\circ\sfX\\
      \Delta_x\zeta&=\rme^{2\lambda t}\,\Delta_y\varphi\circ\sfX&
      B\cdot \nabla_x\zeta&=\rme^{2\lambda t}\,\Big(\tilde B(y,s)\cdot \nabla_y\varphi\Big)\circ\sfX,
    \end{align*}
    where we used the fact that $B=\rme^{\lambda t}\tilde B\circ\sfX$.
    In particular we have
    \begin{displaymath}
      \label{eq:43}
      \partial_t\zeta- B\cdot \nabla_x\zeta=\sfs'\Big(\partial_s\varphi-
      \tilde A(y,s)\cdot \nabla_y\varphi\Big)\circ\sfX
    \end{displaymath}
    We thus have
    \begin{align*}
      \int_\Rd\Big(\partial_t\zeta+\Delta_x\zeta-B\cdot \nabla_x\zeta\Big)\,\d\rho_t&=
      \sfs'(t) \int_\Rd\Big(\partial_s\varphi+\Delta_y\varphi-\tilde A(y,s)\cdot \nabla_y\varphi\Big)\circ\sfX\,\d\rho_t\\&=
      \sfs'(t) \int_\Rd\Big(\partial_s\varphi+\Delta_y\varphi-\tilde A(y,s)\cdot \nabla_y\varphi\Big)\,\d\sigma_{\sfs(t)}
    \end{align*}
    since $\sigma_{\sfs(t)}(E)=\rho_t(\rme^{-\lambda t}\,E)$ for every Borel set $E\subset \Rd$.
    Eventually we obtain
    \begin{displaymath}
      \int_0^{+\infty}\int_\Rd\Big(\partial_t\zeta+\Delta_x\zeta-B\cdot \nabla_x\zeta\Big)\,\d\rho_t\,\d t=
      \int_0^{s_\infty} \int_\Rd \Big(\partial_s\varphi+\Delta_y\varphi-\tilde A(y,s)\cdot \nabla_y\varphi\Big)\,\d\sigma_s\,\d s
    \end{displaymath}
    \eqref{eq:55} follows by a simple application of the change of
    variable formula \eqref{eq:51}, since for every $t>0$
    \begin{align*}
      \int_\Rd \big|\tilde A(y,s)\big|\,\d\sigma_s(y)&\topref{eq:41}=
      \rme^{-\lambda\sft(s)}\int_\Rd \big|A(\rme^{-\lambda
        \sft(s)}y)\big|\,\d\sigma_{s}(y)
      \\&\topref{eq:40}=
      \rme^{-\lambda\sft(s)}\int_\Rd \big|A(x)\big|\,\d\rho_{\sft(s)}(x)=
      \rme^{-\lambda\sft(s)}\int_\Rd \big|B(x)-\lambda x\big|\,\d\rho_{\sft(s)}(x).
    \end{align*}
    Since $\sft'(s)=\rme^{-\lambda \sft(s)}$ we eventually get for
    $t_i=\sft(s_i)$ 
    \begin{align*}
      \int_{s_0}^{s_1}\int_\Rd |\tilde A(x,s)|\,\d\sigma_s\,\d s&=
      \int_{s_0}^{s_1}\Big(\int_\Rd \big|B(x)-\lambda
      x\big|\,\d\rho_{\sft(s)}(x)\Big)\,\sft'(s)\,\d s\\&=
      \int_{t_0}^{t_1} \int_\Rd \big|B(x)-\lambda
      x\big|\,\d\rho_t(x)\,\d t\topref{Cond}<+\infty.
    \end{align*}
    \eqref{eq:3} follows from \eqref{eq:39} when $\varphi$ belongs to
    $C^\infty_{\rmc}(\Rd\times [s_0,s_1])$. 
    If $\varphi\in C^{2,1}_{\rmb}(\Rd\times[s_0,s_1])$
    via a standard convolution and truncation argument we find an
    approximation sequence $\varphi_k \in C^\infty_{\mathrm
      c}(\Rd\times [s_0,s_1])$
    such that $\varphi_k,\partial_t\varphi_k,\nabla\varphi_k,\Delta\varphi_k$
    remains
    uniformly bounded and converge pointwise to
    $\varphi,\partial_t\varphi,\nabla\varphi,\Delta\varphi$
    respectively.
    By \eqref{eq:55} we can apply the Lebesgue Dominated
    Convergence theorem to pass to the limit in \eqref{eq:3} written
    for $\varphi_k$, thus obtaining the same identity for $\varphi$.
  \end{proof}

We conclude this section by a simple remark combining the
regularization technique of Section~\ref{subsec:Bregular} and the time
rescaling \eqref{eq:41}.
\begin{lemma}
  \label{le:bound}
  Let $A:=B-\lambda I$ be a monotone operator, 
  let us consider a sequence $A_{n,m}$,
  $n,m\in \N$, of smooth monotone operators given by Theorem
  \ref{approxB} and Proposition \ref{approxBn},  and let us set
  \begin{equation}
    \label{eq:63}
    \tilde A_{n,m}(y,s):=\rme^{-\lambda\sft(s)}
    A_{n,m}(\rme^{-\lambda\sft(s)}y)\quad y\in \Rd,\ s\in [0,S_\infty)
  \end{equation}
  defined as in \eqref{eq:41}, \eqref{eq:36}.
  Then $\tilde A_{n,m}$ are Lipschitz in $\Rd\times[0,S]$ for every $S\in
  [0,S_\infty)$.
\end{lemma}
\begin{proof}
  We just have to check that $|\partial_s \tilde A_{n,m}(\cdot,s)|$ is
  uniformly bounded in $\Rd\times[0,S]$:
  sine $\sft'(s)=\rme^{-\lambda \sft(s)}$ a simple calculation yields
  \begin{displaymath}
    \partial_s \tilde A_{n,m}(y,s)=-\lambda
    \rme^{-\lambda\sft(s)}\, \tilde
    A_{n,m}(y,s)-\lambda\rme^{-2\lambda \sft(s)}
    \D A_{n,m}(\rme^{-\lambda\sft(s)}y)\cdot y =
    -\lambda\rme^{-\lambda\sft(s)}\tilde Q_{n,m}(y,s)      
  \end{displaymath}
  where
  \begin{displaymath}
    Q_{n,m}(x)=A_{n,m}(x)+\D A_{n,m}(x)\cdot x,\quad x\in \Rd.
  \end{displaymath}
  Since $\rme^{-\lambda\sft(s)}$ is uniformly bounded with all its
  derivative in each compact interval $[0,S]$, $S<\infty$,
  \eqref{eq:34} show that $Q_{n,m}$ is bounded and therefore $\tilde
  A_{n,m}$ is Lipschitz with respect to $s$.
\end{proof}

\section{A comparison result for the backward equation}
In this section we give the proof of Theorem
\ref{thm:crucialintro} in a slightly more general form, in order to be
applied to (a suitably regularized version of) the rescaled
formulation considered in Proposition~\ref{prop:rescaled}.

Let us suppose that $\tilde A:(y,s)\in \Rd\times[0,S_\infty)\to \tilde
A(y,s)\in\Rd$ is a smooth
vector field 
satisfying 
 \begin{gather}
  \label{eq:38}
  \sup_{\Rd\times [0,S]}|\tilde A_s|+|\partial_s \tilde A|+|\D\tilde
  B|<+\infty\quad\text{for every $S\in [0,S_\infty)$,}\\
  \label{eq:45}
  \tilde A(\cdot,s)\text{ is monotone for every }s\in
  [0,S_\infty).
\end{gather}
%
We denote by $\Oper\cdot$ the differential operator defined by
\begin{equation}
  \label{eq:46}
  \Oper\varphi(y,s):=\Delta_y \varphi(y,s)-\tilde A(y,s)\cdot
  \nabla_y\varphi(y,s)\quad
  \varphi(\cdot,s)\in
  C^2(\Rd),\quad
  (y,s)\in \Rd\times [0,S_\infty).
\end{equation}
Thanks to \eqref{eq:38} and \eqref{eq:45}, we can apply the existence result
\cite[Theorem~3.2.1]{Stroock-Varadhan06} and for every $S\in [0,S_\infty)$
and $\phi\in C^\infty_\rmc(\Rd)$ we can find a solution
$\varphi\in C^{2,1}_\rmb(\Rd\times [0,S))$ of the backward evolution equation
\begin{equation} \label{PB:back}
  \partial_s\varphi + \Oper\varphi = 0\quad\text{in } \Rd\times[0,S],
  \quad
  \varphi(\cdot,S)=\phi(\cdot).
\end{equation} 
We have
\begin{teor}
  \label{thm:crucial}
  Let $h:[0,+\infty)\to \R$ be a continuous and non-decreasing function.
  Let $\varphiu,\varphid\in C^{2,1}_{\mathrm b}(\Rd\times [0,S])$
  be solutions of the ``backward'' inequality
  \begin{equation} \label{eq:23}
    \partial_s \varphi+\Oper{\varphi}\ge0\quad
    \text{in }\Rd\times [0,S]
  \end{equation}
such that
\begin{equation*}
  \varphiu(\yu,S)+\varphid(\yd,S)\le h(|y_1-y_2|)\quad\text{for every }\yu,\yd\in \Rd.
\end{equation*}
Then
\begin{equation*}
  \varphiu(\yu,0)+\varphid(\yd,0)\le h(|\yu-\yd|)\quad\text{for every
  }\yu,\yd\in \Rd.
\end{equation*}
\end{teor}
\begin{proof} 
  By approximating $h$ from above, it is not restrictive to assume
  that $h \in C^1[0,+\infty)$ with $h'(0)=0$; in particular the map
  $H(\yu,\yd):= h(|\yu-\yd|)$ is of class $C^1$ in $\Rd\times\Rd$ and
  satisfies
  \begin{gather}
    \label{eq:47}
    \nabla_{\yu}H(\yu,\yd)=-\nabla_{\yd}H(\yu,\yd)=g(\yu,\yd)(\yu-\yd),
    \intertext{where }
    \label{eq:48}
    0\le g(\yu,\yd)=g(\yd,\yu):=
    \begin{cases}
      \frac{h'(|\yu-\yd|)}{|\yu-\yd|}&\text{if }\yu\neq\yd,\\
      0&\text{if }\yu=\yd.
          \end{cases}
  \end{gather}
  
  The argument combines a \variabledoubling technique and a classical variant of the maximum principle.
  Let us first show that if $\varphiu,\varphid$ satisfy the \emph{strict} inequality
\begin{equation} \label{clash<0} 
  \partial_s \varphij+\Oper{\varphij}>0\quad
  \text{in }\Rd\times[0,S),\quad j=1,2.
\end{equation}
then the function
\begin{equation*}
  f(\yu,\yd,s) := \varphiu(\yu,s)+\varphid(\yd,s)-H(\yu,\yd)
\end{equation*}
cannot attains a (local) maximum in a point $(\byu,\byd,\bar s)$ with
$\bar s<S$.
We argue by contradiction and we suppose that $(\byu,\byd,\bar s)$ is
a local maximizer of $f$ with $\bar s<S$; we thus have
\begin{equation*}
  \partial_s f(\byu,\byd,\bar s) \le 0, 
 \quad\nabla_{\yu} f(\byu,\byd,\bar s) = 0,
 \quad\nabla_{\yd} f(\byu,\byd,\bar s) = 0;
\end{equation*}
so that 
\begin{equation} \label{dt<0}
\begin{aligned}
\partial_t \varphiu(\byu,\bar s) + \partial_t \varphid(\byd,\bar s)\le 0
\end{aligned}
\end{equation}
\begin{equation*} 
\begin{aligned}
\nabla_{\yu} \varphiu(\byu,\bar s)
 &= \nabla_{\yu} H(\byu,\byd) \topref{eq:47}= 
 g(\byu,\byd)(\yu-\yd)\\
\nabla_{\yd} \varphid(\byd,\bar s)
&= \nabla_{\yd} H(\byu,\byd) \topref{eq:47}= 
g(\byu,\byd)(\yd-\yu).
\end{aligned} 
\end{equation*}
It follows that 
\begin{equation} \label{Bn<0}
\begin{aligned}
  \tilde A(\byu,\bar s)\cdot\nabla_{\yu}\varphiu(\byu,\bar s)+
  &\tilde A(\byd,\bar s)\cdot \nabla_{\yd}\varphid(\byd,\bar s) \\&=
  g(\byu,\byd)\big(\tilde A(\byu,\bar s)-\tilde A(\byd,\bar
  s)\big)\cdot (\byu-\byd)
  \topref{eq:48}\ge0
\end{aligned}
\end{equation}
On the other hand, since $H(\byu+z,\byd+z) = H(\byu,\byd)$, the function
\begin{displaymath}
  \Rd\ni z\mapsto \varphiu(\byu+z,\bar s)+\varphid(\byd+z,\bar s)-H(\byu,\byd)=
  f(\byu+z,\byd+z,\bar s)
\end{displaymath}
has a local maximum at $z=0$ so that 
\begin{equation} \label{Delta<0}
\begin{aligned}
  \Delta_{\yu}\varphiu(\byu,\bar s)+\Delta_{\yd}\varphid(\byd,\bar s)\le 0.
\end{aligned}
\end{equation}
Combining \eqref{dt<0},\eqref{Bn<0}, and \eqref{Delta<0} we obtain 
\begin{equation*}
  (\partial_s \varphiu+\Oper{\varphiu})(\byu,\bar s)+(\partial_s\varphid+\Oper\varphid)(\byd,\bar s)\le 0,
\end{equation*}
  which contradicts \eqref{clash<0}.
  
Suppose now that $\varphiu,\varphid$ satisfy the inequality \eqref{eq:23} and let us set for $\eps,\delta>0$
\begin{equation*}
  \varphij_{\eps,\delta}(y_j,s):= \varphij(y_j,s)-\delta(S-s)-\eps
  \rme^{-s}|y_j|^2
  \quad j=1,2.
\end{equation*}
We easily get
\begin{align*}
 \partial_s\varphij_{\eps,\delta}&=\partial_s \varphij+\delta +\eps \rme^{-s}|y_j|^2\\
 \Oper{\varphij_{\eps,\delta}}&=\Oper{\varphij}-\rme^{-s}\big(d\eps+2\eps
 \tilde A(y_j,s)\cdot y_j\big)\\
 \partial_s \varphij_{\eps,\delta}+\Oper{\varphij_{\eps,\delta} }
 &\ge\delta +\eps \rme^{-s}\big(|y_j|^2-d-C_n|y_j|\big),
\end{align*}
where $C_n=\sup_{y,s} |\tilde A_n(y,s)|<+\infty$. 
  
It follows that for every $\delta>0$ there exists a coefficient
$\eps>0$ sufficiently small such that
$\varphiu_{\eps,\delta},\varphid_{\eps,\delta}$ satisfy
\eqref{clash<0}.
On the other hand, the continuous function
\begin{equation*}
  (\yu,\yd,s)\mapsto f_{\eps,\delta}(\yu,\yd,s):=\varphiu_{\eps,\delta}(\yu,s)+\varphid_{\eps,\delta}(\yd,s)-
  h(|\yu-\yd|)\quad \yu,\yd\in  \Rd,\ s\in [0,S],
\end{equation*}
attains its maximum at some point $(\byu,\byd,\bar s)\in
\Rd\times\Rd\times [0,S]$;
by the previous argument, we conclude that $\bar s=S$ and therefore
for every $\yu,\yd\in \Rd$
\begin{align*}
  \varphiu_{\eps,\delta}(\yu,0)&+\varphid_{\eps,\delta}(\yd,0)-h(|\yu-\yd|)\le
  f_{\eps,\delta}(\byu,\byd,S)\le \varphiu(\byu,S)+\varphid(\byd,S) - h(|\byu-\byd|)\le 0.      
\end{align*}
  Passing to the limit as $\eps,\delta\downarrow 0$ we conclude.
\end{proof}
We conclude this section by recalling two well known estimates:
\begin{lemma}[Uniform estimates] \label{Grad}
  Let $\varphi\in C^{2,1}_{\mathrm b}(\Rd\times[0,S])\cap
  C^\infty(\Rd\times (0,S))$ be the solution of
  \eqref{PB:back}.
Then
\begin{equation}
\label{eq:49}
\sup_{\Rd\times[0,S]} |\varphi|\le \sup_{ \Rd} |\phi|
,\qquad
\sup_{\Rd\times[0,S]} |\nabla\varphi|\le \sup_{\Rd} |\nabla\phi|.
\end{equation}
\end{lemma}
\begin{proof}
The first inequality is direct application of
the maximum principle (see
e.g.~\cite[Theorem~3.1.1]{Stroock-Varadhan06}.
By differentiating the equation with respect to $y$ we obtain
\begin{equation*}
\partial_s \,\D\varphi+\Oper{\D\varphi}-\D \tilde A\,\D\varphi=0
\end{equation*}
and then
\begin{equation*}
\frac 12\partial_t |\D \varphi|^2+
\frac 12 \Oper{|\D\varphi|^2} - \D \tilde A\D\varphi\cdot\D\varphi - |\D^2 \varphi|^2=0.
\end{equation*}
Since $\tilde A$ is monotone the quadratic form associated to
$\D\tilde A$ is nonnegative and therefore
\begin{equation*}
\partial_t |\D \varphi|^2+
\Oper{|\D\varphi|^2} \ge0.
\end{equation*}
A further application of the maximum principle yields \eqref{eq:49}.
\end{proof}

\section{Proof of Theorem \ref{teor:main}}

We split the proof in various steps.
Just to fix some notation, we consider a family $A_{n,m}$ of smooth, bounded,
Lipschitz, and monotone operators approximating $A:=B-\lambda I$ as
in Proposition
\ref{approxBn} and their rescaled version $\tilde A_{n,m}$ defined by
\eqref{eq:63}.
$\Opernm{n,m}\cdot$ are the associated differential operators
\begin{equation}
  \label{eq:46bis}
  \Opernm{n,m}\varphi(y,s):=\Delta_y \varphi(y,s)-\tilde A_{n,m}(y,s)\cdot
  \nabla_y\varphi(y,s)\quad
  \varphi(\cdot,s)\in
  C^2(\Rd),\quad
  (y,s)\in \Rd\times [0,S_\infty),
\end{equation}
as in \eqref{eq:64}. Lemma \ref{le:bound} show that $\tilde A$
satisfy \eqref{eq:38}.

\noindent\underline{Step 1:} \emph{reduction to the monotone case
}$\lambda=0$.
When $\lambda\neq 0$ we apply the rescaling argument of section \ref{subsec:rescaling}:
we thus introduce the time rescaling $\sft(s)$ defined by
\eqref{eq:36}
and the corresponding 
measures $\sigma^i_s=\tilde\rho^i_{\sft(s)}$ as in \eqref{eq:40},
which satisfy \eqref{eq:55} and \eqref{eq:3} for the rescaled
operators $\tilde A$ of \eqref{eq:41}.
Taking into account Remark \ref{rem:equivalent} and the fact that
$\sigma^i_s=\tilde\rho^i_{\sft(s)}$, the thesis follows if we show
that
\begin{equation}
  \label{eq:59}
  \Cost h{\sigma^1_s}{\sigma^2_s}\le \Cost
  h{\sigma^1_0}{\sigma^2_0}\quad\text{for every }s\in [0,S_\infty),
\end{equation}
(see \eqref{eq:61} for the definition of $S_\infty$).

\noindent\underline{Step 2:} {\em If
  \begin{equation}
    \label{eq:60}
    \Cost h{\sigma^1_{s_1}}{\sigma^2_{s_1}}\le \Cost
    h{\sigma^1_{s_0}}{\sigma^2_{s_0}}\quad\text{for every }0<s_0<s_1<S_\infty,
  \end{equation}
  then \eqref{eq:59} holds.}
When $h$ is bounded, \eqref{eq:60} implies \eqref{eq:59} by taking a
simple limit as $s_0\downarrow0$ and using the fact that the map
$(\sigma^1,\sigma^2)\mapsto \Cost h{\sigma^1}{\sigma^2}$ is continuous
with respect to weak convergence in $\PP(\Rd)\times\PP(\Rd)$.
If \eqref{eq:59} holds for every bounded Lipschitz cost, then it holds
for every {continuous and nondecreasing} \Killed{admissibl}e cost by Lemma \ref{le:cost_approximation}.

\noindent\underline{Step 3:} We claim the following: 

{\em Let $\phiu,\phid\in
  C^\infty_{\mathrm c}(\Rd)$
  be satisfying the constraint $\phiu(\yu)+\phid(\yd)\le h(|\yu-\yd|)$
  Then
\begin{equation}
  \label{eq:64}
  \int_\Rd \phi^1\,\d\sigma^1_{s_1}+\int_\Rd \phi^2\sigma^2_{s_1}\le
  \Cost h{\sigma^1_{s_0}}{\sigma^2_{s_0}}+
  \ell\, K_{n,m}
\end{equation}
where  $\ell:=\sup_\Rd |\nabla\phiu|+\sup_\Rd |\nabla\phid|$ and
\begin{equation*}
  K_{n,m} := \int_{s_0}^{s_1}\int_\Rd  |\tilde A_{n,m}-\tilde A|
  \,\d \sigma^1_s\,\d s+
  \int_{s_0}^{s_1}\int_\Rd  |\tilde A_{n,m}-\tilde A|
  \,\d \sigma^2_s\,\d s.
\end{equation*}}

Indeed, applying 
\cite[Theorem~3.2.1]{Stroock-Varadhan06} 
we can introduce the solutions $\varphiu_{n,m},\varphid_{n,m}\in C^{2,1}_\rmb(\Rd\times[s_0,s_1])$ of the backward equations
\begin{equation*}
  \partial_s \varphij_{n,m}+\Opernm{n,m}\varphij=0\quad\text{in }
  \Rd\times [s_0,s_1],\quad \varphij_{n,m}(\cdot,s_1)=
  \phi^j(\cdot)\quad\text{in }\Rd.
\end{equation*}
Identity~\eqref{eq:3} shows that, for $j=1,2$,
\begin{align*}
  \int_\Rd \varphij(\cdot,s_1)\,\d\sigma^j_{s_1}-\int_\Rd \varphij(\cdot,s_0)\,\d\sigma^j_{s_0} &=
  \int_{s_0}^{s_1}\int_\Rd  \big(\tilde A_{n,m}-\tilde A\big)\cdot
  \nabla\varphij_{n,m}\,\d \sigma^j_s\,\d s
  \\&\topref{eq:49}\le\ell \int_{s_0}^{s_1}\int_\Rd  \big|\tilde A_{n,m}-\tilde A\big|
  \,\d \sigma^j_s\,\d s
\end{align*}
Summing up the these equation for $j=1,2$ we obtain
\begin{equation} \label{eq:4}
  \int_\Rd \varphiu\,\d \sigma^1_{s_1}+
  \int_\Rd \varphid\,\d\sigma^2_{s_1}\le
  \int_\Rd \varphiu_{n,m}(\cdot,s_0)\,\d \sigma^1_{s_0}+
  \int_\Rd \varphid_n(\cdot,s_0)\,\d\sigma^2_{s_0}+\ell K_{n,m}
\end{equation}
Theorem \ref{thm:crucial} yields
$\varphiu_{n,m}(\yu,s_0)+\varphid_{n,m}(\yd,s_0) \le h(|\yu-\yd|)$
which implies \eqref{eq:64}.

\noindent\underline{Step 4:}
\begin{equation}
  \label{eq:65}
  \limsup_{n\up+\infty}\Big(\limsup_{m\up+\infty} K_{n,m}\Big)=0.
\end{equation}
Let us first notice that setting $t_i:=\sft(s_i)$ and recalling that
$\sft'(s)=\rme^{-\lambda \sft(s)}$ we have
\begin{align*}
  \int_{s_0}^{s_1}\int_\Rd  |\tilde A_{n,m}-\tilde A|
  \,\d \sigma^1_s\,\d s&=
  \int_{\sfs(t_0)}^{\sfs(t_1)}\sft'(s)\int_\Rd
  |A_{n,m}-A|\,\d\rho^i_{\sft(s)}\,\d s=
  \int_{t_0}^{t_1}\int_\Rd
  |A_{n,m}-A|\,\d\rho^i_{t}\,\d t  
\end{align*}
so that
\begin{align*}
  K_{n,m}=K_{n,m}^1+K_{n,m}^2,\quad
  K^j_{n,m}:=\int_{t_0}^{t_1}\int_\Rd |A_{n,m}-A|\,\d\rho^j_{t}\,\d
  t\quad j=1,2.
\end{align*}
We can estimate $K^j_{n,m}$ by
\begin{align*}
  K^j_{n,m}\le \int_{t_0}^{t_1}\int_\Rd |A_{n,m}-A_n|\,\d\rho^j_{t}\,\d t  +
  \int_{t_0}^{t_1}\int_\Rd |A_{n}-A|\,\d\rho^j_{t}\,\d t,
\end{align*}
observing that
by \eqref{eq:34}, \eqref{eq:62}, and the Lebesgue Dominated Convergence
Theorem we get
\begin{displaymath}
  \lim_{m\up+\infty}K^j_{n,m}= 
  \int_{t_0}^{t_1}\int_\Rd |A_{n}-A|\,\d\rho^j_{t}\,\d t.
\end{displaymath}
{Since $|A_n(x)|\le |\sfA^\circ(x)|\le |A(x)|=|B(x)-\lambda x|$ for every $x\in
  \Rd$, the integrability assumption \eqref{Cond},} a further application of the Lebesgue Theorem, and \eqref{B approx}
yield
\begin{equation}
  \label{eq:72}
  \lim_{n\up+\infty}\Big(\lim_{m\up+\infty} K_{n,m}\Big)=
   \int_{t_0}^{t_1}\int_\Rd |\sfA^\circ-A|\,\d\rho^j_{t}\,\d t.
\end{equation}

This last  integrand is $0$ if $A$ coincides with the minimal
selection of $\sfA$, in particular when $A$ is continuous. In the
general case, \Killed{applying} the regularity result of
\cite{Bogachev-Krylov-Rockner01}
shows that $\rho^j_t\ll \Leb d$ for $\Leb 1$ a.e.\ $t\in (0,+\infty)$
and \eqref{eq:26} says that $\sfA^\circ=A$ $\Led$-a.e.\ in $\Rd$;
therefore the last integral of \eqref{eq:72} vanishes and we get \eqref{eq:65}.
%
%

\noindent\underline{Step 5:} \emph{conclusion.}

Thanks to \eqref{eq:65}, passing to the limit in \Killed{\eqref{eq:4}} {\eqref{eq:64}} we obtain
\begin{equation*}
  \int_\Rd \phiu\,\d \sigma^1_{s_1}+\int_\Rd \phid\,\d\sigma^2_{s_1}\le\Cost h{\sigma^1_{s_0}}{\sigma^2_{s_0}}.
\end{equation*}
Taking the supremum with respect to $\phiu,\phid\in
C^\infty_\rmc(\Rd)$ and recalling Proposition \ref{Kant+}
we obtain \eqref{eq:60}.
\begin{remark}
  \upshape
  As it appears from the final argument of the previous step 4, in the case
  when $A=B-\lambda I$ is the minimal selection $\sfA^\circ$ of $\sfA$ (in
  particular when $B$ is continuous), we do not need to invoke the
  regularity result of \cite{Bogachev-Krylov-Rockner01} to conclude
  our proof.
\end{remark}

\begin{proof}[Proof of Corollary \ref{cor:main}]
For (a), it is sufficient to observe that $e^{\lambda t} \ge 1$; this implies $h(r) \le h_{\lambda t}(r)$ and so 
\begin{equation*}
 \Cost h{\mu^1_t}{\mu^2_t} \le \Cost {h_{\lambda t}}{\mu^1_t}{\mu^2_t} 
 \topref{++}\le \Cost h{\mu^1_0}{\mu^2_0}.
\end{equation*}
Similarly, for (a) and (b)
\begin{equation*}
  e^{p\lambda t}\Cost h{\mu^1_t}{\mu^2_t}\le\Cost {h_{\lambda t}}{\mu^1_t}{\mu^2_t} 
 \topref{++}\le \Cost h{\mu^1_0}{\mu^2_0}.
\end{equation*}
We conclude recalling that 
\begin{equation*}
 W_p(\mu^1,\mu^2) = \Cost h{\mu^1}{\mu^2}^{1/p} \mbox{ with } h(r) = {|r|}^p 
\end{equation*}
and applying (a) and (b).
\end{proof}
       
\def\cprime{$'$} \def\cprime{$'$}

\end{document}